%% file: main.tex
\documentclass[a4paper,reqno]{amsart}

\usepackage[foot]{amsaddr}

\usepackage{amsmath,amssymb,amsfonts,amsthm}
\usepackage[colorlinks=true,citecolor=blue]{hyperref}
\usepackage[utf8]{inputenc}
\usepackage{enumerate}
\usepackage{booktabs}
\usepackage{tikz,calc}
\usepackage{thmtools}
\usepackage{thm-restate}
\usepackage{cleveref}

\usepackage[
maxbibnames=99,
style=ext-numeric,
backend=biber,
articlein=false,
sorting=nyt,
sortcites,
natbib=false,
doi=false,
url=false,
isbn=false,
eprint=false]{biblatex}

\theoremstyle{plain}
\newtheorem{theorem}{Theorem}

\newtheorem{lemma}{Lemma}[section]

\theoremstyle{definition}
\newtheorem{definition}{Definition}
\newtheorem{example}{Example}

\newtheorem{problem}{Problem}

\newcommand{\abs}[1]{\left\lvert#1\right\rvert}

\let\leq\leqslant
\let\geq\geqslant
\let\phi\varphi

\title{Independent domination in the graph defined by\\ two consecutive levels of the $n$-cube}

\author[T. Kalinowski]{Thomas Kalinowski$^1$}
\author[U. Leck]{Uwe Leck$^2$}

\address{$^1$Institut f\"ur Mathematik, Universit\"at Rostock, Germany}
\address{$^2$Institut f\"ur Mathematik, Europa-Universit\"at Flensburg, Germany}

\email[T.~Kalinowski]{thomas.kalinowski@uni-rostock.de}
\email[U.~Leck]{Uwe.Leck@uni-flensburg.de}

\date{\today}

\begin{document}

\begin{abstract}
Fix a positive integer $n$ and consider the bipartite graph whose vertices are the $3$-element subsets and the $2$-element subsets of $[n]=\{1,2,\dots,n\}$, and there is an edge between $A$ and $B$ if $A\subset B$. We prove that the domination number of this graph is $\binom{n}{2}-\lfloor\frac{(n+1)^2}{8}\rfloor$, we characterize the dominating sets of minimum size, and we observe that the minimum size dominating set can be chosen as an independent set. This is an exact version of an asymptotic result by Balogh, Katona, Linz and Tuza (2021). For the corresponding bipartite graph between the $(k+1)$-element subsets and the $k$-elements subsets of $[n]$ ($k\geq 3$), we provide a new construction for small independent dominating sets. This improves on a construction by Gerbner, Kezegh, Lemons, Palmer, P\'alv\"olgyi and Patk\'os (2012), who studied these independent dominating sets under the name saturating flat antichains.

\smallskip


\end{abstract}

\maketitle

\input{1-intro}
\input{2-dom32}
\input{3-asymp}

\input{4-open_problems.tex}
\printbibliography

\input{5-appendix}

\end{document}

%% file: 1-intro.tex
\section{Introduction}\label{sec:intro}
For a graph $G=(V,E)$, a vertex set $D\subseteq V$ is called \emph{dominating} if every vertex $v\in V\setminus D$ has a neighbor in $D$. The domination number $\gamma(G)$ is the minimum size of a dominating set in $G$. Adding the condition that the set $D$ is an \emph{independent set} in $G$, that is, no two vertices in $D$ are adjacent, we obtain the \emph{independent domination number}, denoted by $i(G)$.
These graph parameters have been studied intensively~\cite{Haynes2013,Haynes2017,Goddard2013}. Clearly, $\gamma(G)\leq i(G)$, and it is known that $\gamma(G)=i(G)$ for claw-free graphs $G$~\cite{Allan1978}.

Let $n$ be a positive integer. We use $[n]$ to denote the set $\{1,2,\dots,n\}$, and for a positive integer $k<n$, $\binom{[n]}{k}$ denotes the set of $k$-element subsets of $[n]$. For $1\leq k<l\leq n$, let $G_{l,k}$ be the bipartite graph with vertex set $V=\binom{[n]}{k}\cup\binom{[n]}{l}$ where $A\in\binom{[n]}{k}$ and $B\in\binom{[n]}{l}$ are adjacent if $A\subseteq B$. The investigation of the domination number of $G_{l,k}$ has been initiated in~\cite{Badakhshian_2019,Balogh2021}. In this paper we focus on the case $l=k+1$. For $(l,k)=(3,2)$ we extend the arguments from~\cite{Gruettmueller2009} to confirm a conjecture from \cite{Badakhshian_2019}, an asymptotic version of which has been established in~\cite{Balogh2021} (for general $(l,2)$).
\nocite{Gerbner2012}
\begin{theorem}\label{thm:23}
  $\displaystyle i(G_{3,2})=\gamma(G_{3,2})=\binom{n}{2}-\left\lfloor\frac{(n+1)^2}{8}\right\rfloor$.
\end{theorem}
For a dominating set in $D$ in $G_{3,2}$ we write $D_2$ and $D_3$ for the collections of $2$-sets and $3$-sets in $D$, respectively, that is $D_i=D\cap\binom{[n]}{i}$ for $i\in\{2,3\}$. In the proof of Theorem \ref{thm:23} it turns out to be useful to associate with a dominating set $D=D_2\cup D_3$ a graph $H=H(D)$ with vertex set $V(H)=[n]$. Two vertices $x$ and $y$ are adjacent in $H$ if $\{x,y\}\notin D$, but $\{x,y,z\}\in D$ for some $z\in[n]$. In other words, the edge set of $H(D)$ is $E(H)=\Delta D_3\setminus D_2$ where $\Delta$ denotes the \emph{shadow} defined by $\Delta D_3=\{A\,:\,\abs{A}=2\text{ and }A\subseteq X\text{ for  some }X\in D_3\}$. From $D$ being a dominating set it follows that the vertex set of every triangle of the graph $H$ is an element of $D_3$. Moreover, if every edge of $H$ is contained in a triangle in $H$ (in other words, for every edge $xy$, the vertices $x$ and $y$ have a common neighbor $z$), and $D$ is a minimal dominating set, then $D_3$ is the set of triangles in $H$ and $D_2$ is the set of non-edges: $D_2=\binom{[n]}{2}\setminus E(H)$. In this situation, $D$ is an independent set in $G_{3,2}$. 

We will show that the graphs coming from minimum dominating sets in $G_{3,2}$ can be described
as follows. For $1<s<n$, let $K^+_{s,n-s}$ be the graph such that
 $[s]\subseteq[n]$ induces a matching of size $\lfloor s/2\rfloor$ (and an additional isolated vertex if $s$ is odd),
$[s+1,n]$ is an independent set, and
$uv\in E\left(K^+_{s,n-s}\right)$ for all $u\in[s]$, $v\in[s+1,n]$ (see Figure~\ref{fig:standard_construction}). Moreover, let $H_{5a}$, $H_{5b}$ and $H_9$ be the three graphs shown in Figure~\ref{fig:extra_graphs}.
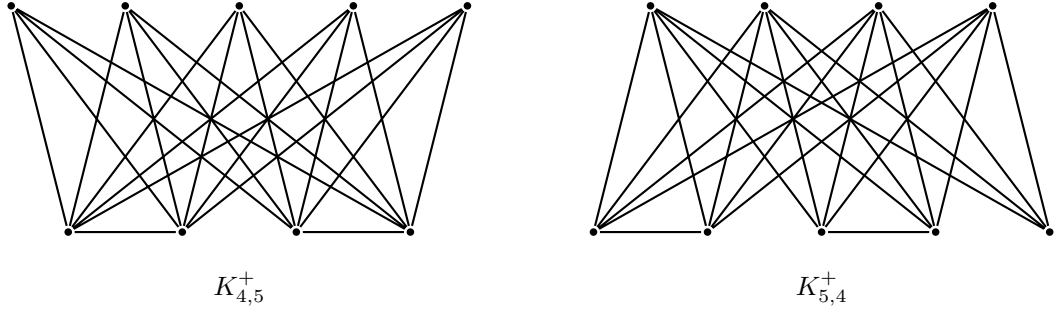
\begin{figure}[htb]
  \begin{minipage}[b]{.45\linewidth}
    \centering
    \begin{tikzpicture}[scale=1.5]
      \node[circle,fill=black,outer sep=1pt,inner sep=1pt] (v1) at (0,0) {};
      \node[circle,fill=black,outer sep=1pt,inner sep=1pt] (v2) at (1,0) {};
      \node[circle,fill=black,outer sep=1pt,inner sep=1pt] (v3) at (2,0) {};
      \node[circle,fill=black,outer sep=1pt,inner sep=1pt] (v4) at (3,0) {};
      \node[circle,fill=black,outer sep=1pt,inner sep=1pt] (w1) at (-.5,2) {};
      \node[circle,fill=black,outer sep=1pt,inner sep=1pt] (w2) at (.5,2) {};
      \node[circle,fill=black,outer sep=1pt,inner sep=1pt] (w3) at (1.5,2) {};
      \node[circle,fill=black,outer sep=1pt,inner sep=1pt] (w4) at (2.5,2) {};
      \node[circle,fill=black,outer sep=1pt,inner sep=1pt] (w5) at (3.5,2) {};
      \draw[thick] (v1) -- (v2) -- (w1) -- (v1) -- (w2) -- (v2) -- (w3) -- (v1) -- (w4) -- (v2) --(w5) -- (v1);
      \draw[thick] (v3) -- (v4) -- (w1) -- (v3) -- (w2) -- (v4) -- (w3) -- (v3) -- (w4) -- (v4) --(w5) -- (v3); 
      \node[draw=none,fill=none] at (1.5,-0.5) {$K^+_{4,5}$};
    \end{tikzpicture}
  \end{minipage}
  \begin{minipage}[b]{.45\linewidth}
    \centering
    \begin{tikzpicture}[scale=1.5]
      \node[circle,fill=black,outer sep=1pt,inner sep=1pt] (v1) at (-.5,0) {};
      \node[circle,fill=black,outer sep=1pt,inner sep=1pt] (v2) at (.5,0) {};
      \node[circle,fill=black,outer sep=1pt,inner sep=1pt] (v3) at (1.5,0) {};
      \node[circle,fill=black,outer sep=1pt,inner sep=1pt] (v4) at (2.5,0) {};
      \node[circle,fill=black,outer sep=1pt,inner sep=1pt] (v5) at (3.5,0) {};
      \node[circle,fill=black,outer sep=1pt,inner sep=1pt] (w1) at (0,2) {};
      \node[circle,fill=black,outer sep=1pt,inner sep=1pt] (w2) at (1,2) {};
      \node[circle,fill=black,outer sep=1pt,inner sep=1pt] (w3) at (2,2) {};
      \node[circle,fill=black,outer sep=1pt,inner sep=1pt] (w4) at (3,2) {};
      \draw[thick] (v1) -- (v2) -- (w1) -- (v1) -- (w2) -- (v2) -- (w3) -- (v1) -- (w4) -- (v2);
      \draw[thick] (v3) -- (v4) -- (w1) -- (v3) -- (w2) -- (v4) -- (w3) -- (v3) -- (w4) -- (v4);
      \draw[thick] (w1) -- (v5) -- (w2);
      \draw[thick] (w3) -- (v5) -- (w4);
      \node[draw=none,fill=none] at (1.5,-0.5) {$K^+_{5,4}$};
    \end{tikzpicture}
  \end{minipage}
  \caption{The graphs $K^+_{4,5}$ and $K_{5,4}^+$}
  \label{fig:standard_construction}
\end{figure}

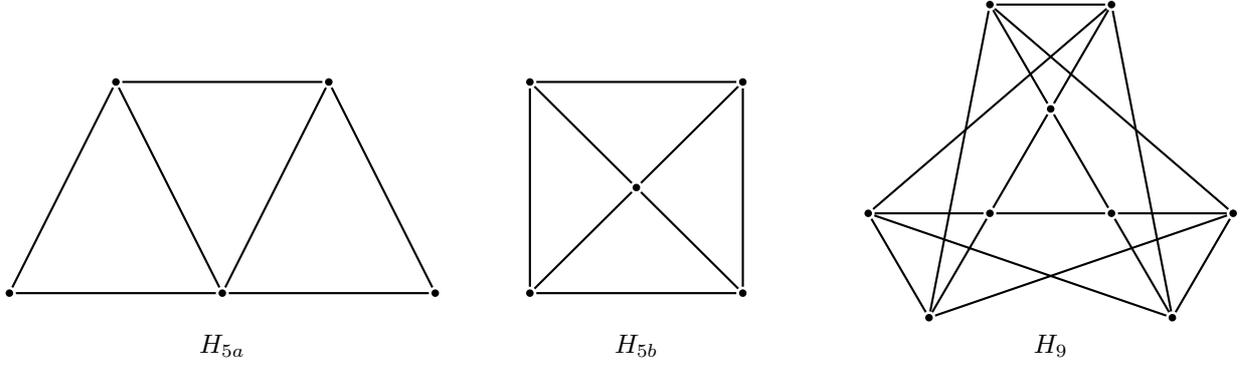
\begin{figure}[htb]
   \begin{minipage}[b]{.35\linewidth}
    \centering
    \begin{tikzpicture}[scale=1.4]
      \node[circle,fill=black,outer sep=1pt,inner sep=1pt] (v1) at (0,0) {};
      \node[circle,fill=black,outer sep=1pt,inner sep=1pt] (v2) at (2,0) {};
      \node[circle,fill=black,outer sep=1pt,inner sep=1pt] (v3) at (4,0) {};
      \node[circle,fill=black,outer sep=1pt,inner sep=1pt] (v4) at (1,2) {};
      \node[circle,fill=black,outer sep=1pt,inner sep=1pt] (v5) at (3,2) {};
      \draw[thick] (v1) -- (v2) -- (v3) -- (v5) -- (v4) -- (v1);
      \draw[thick] (v4) -- (v2) -- (v5);
      \node at (2,-0.5) {$H_{5a}$};
    \end{tikzpicture}
  \end{minipage}\hfill
   \begin{minipage}[b]{.29\linewidth}
    \centering
    \begin{tikzpicture}[scale=1.4]
      \node[circle,fill=black,outer sep=1pt,inner sep=1pt] (v1) at (0,0) {};
      \node[circle,fill=black,outer sep=1pt,inner sep=1pt] (v2) at (2,0) {};
      \node[circle,fill=black,outer sep=1pt,inner sep=1pt] (v3) at (2,2) {};
      \node[circle,fill=black,outer sep=1pt,inner sep=1pt] (v4) at (0,2) {};
      \node[circle,fill=black,outer sep=1pt,inner sep=1pt] (v5) at (1,1) {};
      \draw[thick] (v1) -- (v2) -- (v3) -- (v4) -- (v1) -- (v5) -- (v2);
      \draw[thick] (v3) -- (v5) -- (v4);
      \node at (1,-0.5) {$H_{5b}$};
    \end{tikzpicture}
  \end{minipage}\hfill
  \begin{minipage}[b]{.35\linewidth}
    \centering
    \begin{tikzpicture}[scale=.8]
      \node[circle,fill=black,outer sep=1pt,inner sep=1pt] (v1) at (0,0) {};
      \node[circle,fill=black,outer sep=1pt,inner sep=1pt] (v2) at (2,0) {};
      \node[circle,fill=black,outer sep=1pt,inner sep=1pt] (v3) at (4,0) {};
      \node[circle,fill=black,outer sep=1pt,inner sep=1pt] (v4) at (6,0) {};
      \node[circle,fill=black,outer sep=1pt,inner sep=1pt] (v5) at (1,-1.73) {};
      \node[circle,fill=black,outer sep=1pt,inner sep=1pt] (v6) at (5,-1.73) {};
      \node[circle,fill=black,outer sep=1pt,inner sep=1pt] (v7) at (3,1.73) {};
      \node[circle,fill=black,outer sep=1pt,inner sep=1pt] (v8) at (2,3.46) {};
      \node[circle,fill=black,outer sep=1pt,inner sep=1pt] (v9) at (4,3.46) {};
      \draw[thick] (v1) -- (v2) -- (v3) -- (v4) -- (v6) -- (v3) -- (v7) -- (v2) -- (v5) --(v1) --
      (v9) -- (v8) -- (v7) -- (v9) -- (v6) --(v1);
      \draw[thick] (v4) -- (v5) -- (v8) -- (v4);
      \node at (3,-2.2) {$H_{9}$};
    \end{tikzpicture}
    \end{minipage}
    \caption{Three small graphs}
    \label{fig:extra_graphs}
\end{figure}

\begin{theorem}\label{thm:extremal_constructions}
 Suppose $D=D_2\cup D_3$ is a dominating set in $G_{3,2}$ with
 $D=\binom{n}{2}-\left\lfloor\frac{(n+1)^2}{8}\right\rfloor$ and let $H=H(D)$ be the corresponding graph. 
 \begin{enumerate}
 \item If $n\equiv 0\pmod 4$ then $H$ is isomorphic to $ K^+_{n/2,n/2}$.
 \item If $n\equiv 1\pmod 4$ then $H$ is isomorphic to one of the following graphs: $H_{5a}$, $H_{5b}$, $H_9$,
   $K^+_{s,n-s}$ for $2s\in\{n-1,n+1,n+3\}$.
 \item If $n\equiv 2\pmod 4$ then $H$ is isomorphic to $K^+_{(n+2)/2,(n-2)/2}$.
 \item If $n\equiv 3\pmod 4$ then $H$ is isomorphic to $ K^+_{(n+1)/2,(n-1)/2}$.  
 \end{enumerate}
\end{theorem}
All but one of the graphs in Theorem~\ref{thm:extremal_constructions} have the property that every edge is contained in a triangle, so that the corresponding dominating set $D$ is determined by the graph $H$, and is an independent set in $G_{3,2}$. The exception is the graph $K^+_{(n+1)/2,(n-1)/2}$ for $n\equiv 1\pmod 4$. In this graph the edges that are not contained in a triangle form a star $S$, say with center $v$, with $\frac{n-1}{2}$ leaves, and a corresponding minimum dominating set in $G_{3,2}$ is obtained by taking $D_2$ to be the set of non-edges and $D_3$ the set of triangles together with $\frac{n-1}{4}$ sets of the form $\{v,x,y\}$ where $x$ and $y$ are leaves of $S$ (see the proof of Theorem \ref{thm:extremal_constructions} for a precise statement).

For $k\geq 3$, \cite[Theorem 14]{Gerbner2012} bounds the minimum size of an independent dominating set in $G_{k+1,k}$ as follows:
\begin{equation}\label{eq:gerbner_bound}
  \left(1-\frac{k-1}{k}t_k-o(1)\right)\binom{n}{k}\leq
  i(G_{k+1,k})\leq\left(1-\frac12\left(\frac{k-1}{k}\right)^{k-1}+o(1)\right)\binom{n}{k}.  
\end{equation}

Here, the $t_k$ is the Tur\'an-density for the complete $k$-uniform hypergraph on $k+1$ vertices, that is, 
\[t_k=\lim_{n\to\infty}\frac{\operatorname{ex}(n,K^{(k)}_{k+1})}{\binom{n}{k}},\]
where $\operatorname{ex}(n,K^{(k)}_{k+1})$ is the maximal number of edges in a $k$-uniform hypergraph on $n$ vertices without a complete $k$-uniform hypergraph on $k+1$ vertices as a subhypergraph. We improve the upper bound in~\eqref{eq:gerbner_bound} as follows.
\begin{theorem}\label{thm:independent_dominating_asymp}
  For every $k\geq 3$ and every $\alpha$ with $0<\alpha<1$,
  \[i\left(G_{k+1,k}\right)\leq\left(1-\frac{(k-1)^2\alpha(1-\alpha)^{k-2}}{k}+o(1)\right)\binom{n}{k}.\]
  This bound is minimized for $\alpha$ being the root of the polynomial $(k-1)x^k-kx+1=0$ in the interval $[0,\frac12]$.
\end{theorem}
For small $k$, the values for the bounds are collected in Table~\ref{tab:numbers}, where we used the
upper bounds for $t_k$ from~\cite{Chung1999} for odd $k$,\textcolor{red}{~\cite{Markstroem2009} for $k=4$, and~\cite{Lu2009} for $k=6$.}
\begin{table}[htb]
  \centering
  \caption{Numerical values of the asymptotic bounds for
    $i(G_{k,k+1})/\binom{n}{k}$.}\label{tab:numbers}
  \begin{tabular}{cccc} \toprule
    $k$ & lower bound & upper bound from \cite{Gerbner2012} & new upper bound \\ \midrule
    3 & $0.604$ & $0.778$ & $0.691$ \\
    4 & $0.447$ & $0.790$ & $0.683$ \\
    5 & $0.384$ & $0.796$ & $0.673$ \\
    6 & $0.305$ & $0.800$ & $0.666$ \\
    7 & $0.279$ & $0.802$ & $0.661$ \\ \bottomrule
  \end{tabular}
\end{table}
\subsection*{Notation} Throughout we consider only simple graphs. We write $xy$ for the edge
$\{x,y\}$ and $xyz$ for a triangle with vertices $x$, $y$ and $z$. For a vertex $x$, $d(x)$ is the
degree of $x$ and $N(x)$ is the neighborhood of $x$. With $t(x)$ and $t(xy)$ we denote the number of
triangles containing the vertex $x$ and the edge $xy$, respectively. For a set $X$ of vertices in a
graph $G=(V,E)$, $G[X]$ denotes the subgraph induced by $X$ and $E(X)$ denotes the set of edges with
both endpoints in $X$. Similarly, for $X,Y\subseteq V$, $E(X,Y)$ is the set of edges with one
endpoint in $X$ and one endpoint in $Y$. For a set $A$, its \emph{shadow} $\Delta A$ is defined as $\Delta A=\{B\,:\,B\subseteq A\text{ and }\abs{B}=\abs{A}-1\}$, and for a family $\mathcal F$ of $k$-sets, $\Delta\mathcal F=\bigcup_{A\in\mathcal F}\Delta A=\{B\,:\,\abs{B}=k-1\text{ and }B\subseteq A\text{ for some }A\in\mathcal F\}$. 


%% file: 2-dom32.tex
\section{The domination number of \texorpdfstring{$G_{3,2}$}{}}\label{sec:32}
In this section we prove \Cref{thm:23,thm:extremal_constructions}. The statements for independent domination are just a rewording of the main results in~\cite{Gruettmueller2009}, and our approach is to adjust the arguments from this paper so that they apply without the independence assumption. The proof in~\cite{Gruettmueller2009} is based on the observation that there is a one-to-one correspondence between independent dominating sets in $G_{3,2}$ and graphs with the property that every edge is contained in a triangle. More precisely, given such a graph $H$ the dominating set in $G_{3,2}$ consists of the 3-sets which are triangles in $H$ and the $2$-sets which are not edges in $H$. Conversely, for a given independent dominating set $D=D_2\cup D_3$, we
obtain the associated graph $H$ by setting $E(H)=\Delta D_3$. Dropping the independence assumption, we still define a graph associated with a dominating set as follows.
\begin{definition}
  Let $D=D_2\cup D_3$ be a dominating set in $G_{3,2}$. The graph $H=H(D)$ has vertex set $[n]$ and edge set $E(H)=\Delta D_3\setminus D_2$.
\end{definition}
The graph $H$ does not necessarily have the property that every edge is contained in a triangle, and going from $D$ to $H(D)$ we are losing some information: in general, it is not possible to reconstruct $D$ from $H$. Nevertheless, we can bound $\abs{D}$ in terms of $H$ and this bound turns out to be sufficient
to prove \Cref{thm:23,thm:extremal_constructions}. 
\begin{lemma}\label{lem:translation}
  Let $H=(V,E)$ be the graph associated with a minimal dominating set $D=D_2\cup D_3$. Let $T$ be the set of triangles in $H$, and let $E_0$ be the set of edges that are not contained in a triangle. Then
  \[\abs{D}\geq\binom{n}{2}-\left(\abs{E}-\abs{T}-\left\lceil\frac12\abs{E_0}\right\rceil\right).\]
\end{lemma}
\begin{proof}
  Since $D$ is a dominating set, $T\subseteq D_3$ and $\binom{[n]}{2}\setminus\Delta D_3\subseteq D_2$.  Moreover, for every $xy\in E_0$ there must be a 3-set in $D_3\setminus T$ containing $x$ and $y$. Since any such 3-set covers at most 2 elements of $E_0$, we obtain $\abs{D_3\setminus
  T}\geq\left\lceil\frac12\abs{E_0}\right\rceil$. Putting everything together,
  \[\abs{D}=\abs{D_3}+\abs{D_2}=\abs{T}+\abs{D_3\setminus T}+\binom{n}{2}-\abs{\Delta D_3}+\abs{D_2\cap\Delta D_3}\geq \abs{T}+\left\lceil\frac12\abs{E_0}\right\rceil+\binom{n}{2}-\abs{E}.\qedhere\]
\end{proof}
As a consequence, an upper bound on $\abs{E}-\abs{T}-\left\lceil\frac12\abs{E_0}\right\rceil$ for arbitrary $H$ implies a lower bound for $\gamma(G_{3,2})$. Note that we only need the lower bound because $\gamma(G_{3,2})\leq i(G_{3,2})\leq\binom{n}{2}-\left\lfloor\frac{(n+1)^2}{8}\right\rfloor$ is immediate by looking at the dominating sets corresponding to the graphs listed in \Cref{thm:extremal_constructions}. 
\begin{lemma}\label{lem:main_inequality}
  Let $H=(V,E)$ be a graph with vertex set $V=[n]$. Let $T$ be the set of triangles in $H$, and let $E_0$ be the set of edges that are not contained in a triangle. Then
  \begin{equation}\label{eq:edges_minus_triangles_bound}
    \abs{E}-\abs{T}-\frac12\abs{E_0}\leq\left\lfloor\frac{(n+1)^2}{8}\right\rfloor.
  \end{equation}
  Moreover, equality is possible only if $E_0=\emptyset$ or ($n\equiv 1\pmod 4$ and $\abs{E_0}$ is even).
\end{lemma}
\begin{proof}
  Fix a triangle $xyz\in T$. For $i\in\{0,1,2,3\}$, let $a_i$ be the number of vertices in $V\setminus\{x,y,z\}$ with exactly $i$ neighbors in $\{x,y,z\}$. Then
  \begin{align}
      a_0+a_1+a_2+a_3 &= n-3,\label{eq:vertices}\\
      a_1+2a_2+3a_3 &= (d(x)-2)+(d(y)-2)+(d(z)-2),\label{eq:edges}\\
      a_2+3a_3 &=  (t(xy)-1)+(t(yz)-1)+(t(xz)-1).\label{eq:triangles}
  \end{align}
  where \eqref{eq:vertices} comes from counting the vertices in $V\setminus\{x,y,z\}$, \eqref{eq:edges} from counting the edges between $\{x,y,z\}$ and $V\setminus\{x,y,z\}$, and \eqref{eq:triangles} from counting the triangles with exactly two vertices in $\{x,y,z\}$. The difference of \eqref{eq:edges} and \eqref{eq:triangles} provides an expression for $a_1+a_2$, and equating this with $a_1+a_2=n-3-(a_0+a_3)$ (from \eqref{eq:vertices}), we get
 \[(d(x)-2) + (d(y)-2) + (d(z)-2)\\ - (t(xy)-1) - (t(yz)-1) - (t(xz)-1)  = n - 3 - (a_0+a_3), \]
  We set $\alpha_{xyz}=a_0+a_3$ (the number of vertices in $V\setminus\{x,y,z\}$ which have 0 or 3 neighbors in $\{x,y,z\}$) and rearrange the above to obtain
  \[d(x)+d(y)+d(z)-(t(xy)-1)-(t(yz)-1)-(t(xz)-1)=n+3-\alpha_{xyz}.\]
  Taking the sum over $T$ and setting $\alpha=\sum_{xyz\in T}\alpha_{xyz}$,
  \[\sum_{x\in V}t(x)d(x)-\sum_{xy\in E}t(xy)\left(t(xy)-1\right)=(n+3)\abs{T}-\alpha\]  
  Substituting
  \[t(x)=\frac12\sum_{y\in N(x)}t(xy)=\frac12d(x)+\frac12\sum_{y\in N(x)}\left(t(xy)-1\right)\]
  and subtracting $\sum_{xy\in E}\left(t(xy)-1\right)=3\abs{T}-\abs{E}$, we obtain
  \begin{equation}\label{eq:first_step}
    \frac12\sum_{x\in V}d(x)^2+\sum_{xy\in E}\left(t(xy)-1\right)\left(\frac{d(x)+d(y)}{2}-t(xy)-1\right)=n\abs{T}+\abs{E}-\alpha.
  \end{equation}
  For every $xy\in E$
  \[t(xy)=\abs{N(x)\cap N(y)}\leq\min\{d(x)-1,d(y)-1\}\leq\frac{d(x)+d(y)}{2}-1.\]
  As a consequence, in the second sum on the left-hand side of~(\ref{eq:first_step}), only the terms for $xy\in E_0$ can be negative, and it follows that
  \[\frac12\sum_{x\in V}d(x)^2-\sum_{xy\in E_0}\left(\frac{d(x)+d(y)}{2}-1\right)=n\abs{T}+\abs{E}-\alpha-\beta\]
  where (with $E_1=E\setminus E_0$ being the set of edges contained in at least one triangle)
  \[\beta=\sum_{xy\in E_1}\left(t(xy)-1\right)\left(\frac{d(x)+d(y)}{2}-t(xy)-1\right)\geq 0.\]
  With $d(x)+d(y)\leq n$ for every $xy\in E_0$, we obtain
  \begin{equation}\label{eq:second_step}
    \frac12\sum_{x\in V}d(x)^2-\frac{n}{2}\abs{E_0}\leq n\abs{T}+\abs{E}-\abs{E_0}-\alpha-\beta.
  \end{equation}
  Next we substitute
  \begin{multline*}
    \sum_{x\in V}d(x)^2=\sum_{x\in
      V}\left[(n+1)d(x)-\frac{(n+1)^2}{4}+\left(\frac{n+1}{2}-d(x)\right)^2\right]\\
    =2(n+1)\abs{E}-\frac{n(n+1)^2}{4}+\sum_{x\in V}\left(\frac{n+1}{2}-d(x)\right)^2
  \end{multline*}
  into~(\ref{eq:second_step}) and rearrange the result:
  \begin{equation} \label{eq:final_inequality}
    \abs{E}-\abs{T}-\frac12\abs{E_0}\leq\frac{(n+1)^2}{8}-\frac{1}{n}\left(\abs{E_0}+\alpha+\beta+\frac12\gamma\right)
  \end{equation}
  with $\gamma=\sum_{x\in V}\left(\frac{n+1}{2}-d(x)\right)^2\geq 0$. For $n\equiv 3\pmod 4$, $\frac{(n+1)^2}{8}$ is an integer, hence \eqref{eq:edges_minus_triangles_bound} follows with $\abs{E_0}=\alpha=\beta=\gamma=0$ being necessary for equality. For $n\equiv 0\pmod 2$, $\left\lfloor\frac{(n+1)^2}{8}\right\rfloor=\frac{(n+1)^2}{8}-\frac18$, and with 
  \[\gamma=\sum_{x\in V}\left(\frac{n+1}{2}-d(x)\right)^2\geq\frac{n}{4},\]
  inequality \eqref{eq:edges_minus_triangles_bound} follows from \eqref{eq:final_inequality}, and $\abs{E_0}=\alpha=\beta=0$ is necessary for equality. For $n\equiv 1\pmod 4$, if $E_0=\emptyset$, then \eqref{eq:edges_minus_triangles_bound} follows by rounding \eqref{eq:final_inequality}. If $E_0\neq \emptyset$, then the right-hand side of \eqref{eq:final_inequality} is strictly less than $\frac{(n+1)^2}{8}=\left\lfloor\frac{(n+1)^2}{8}\right\rfloor+\frac12$, and \eqref{eq:edges_minus_triangles_bound} follows with a strict inequality if $\abs{E_0}$ is odd.
 \end{proof}
\Cref{thm:23} is an immediate consequence of \Cref{lem:translation,lem:main_inequality}. To prove \Cref{thm:extremal_constructions}, we start with the observation that the optimal independent dominating sets have been described completely in~\cite{Gruettmueller2009}. It remains to check whether there are any additional dominating sets of size $\gamma(G_{3,2})$, such that $E_0\neq\emptyset$ in the associated graph.
\begin{proof}[Proof of \Cref{thm:extremal_constructions}]
  Let $H$ be the graph associated with $D$. For $n\not\equiv 1\pmod 4$, $E_0=\emptyset$, so that $D$ is an independent set, and the result follows from~\cite{Gruettmueller2009}. The case $n\equiv 1\pmod 4$ remains. That this might get a bit more complicated is indicated by the observation that this is the only case where one of the graphs listed in \Cref{thm:extremal_constructions} satisfies $E_0\neq\emptyset$: For $H=K_{(n+1)/2,(n-1)/2}$ there is one isolated vertex, say vertex 1, in the subgraph induced by $1,2,\dots,(n+1)/2$, and in $H$ this vertex is adjacent to every vertex in the independent set $\{(n+3)/2,\dots,n\}$. So $E_0=\{1x\,:\,x=(n+3)/2,\dots,n\}$, and an optimal dominating set corresponding to $H$ is given by
  \begin{align*}
    D_3 &= T\cup\left\{1xy\,:\,x=(n+1)/2+i,\,y=n-i,\,i\in\{1,\dots,(n-1)/4\}\right\},\\
    D_2 &= \binom{[n]}{2}\setminus E\cup\left\{xy\,:\,x=(n+1)/2+i,\,y=n-i,\,i\in\{1,\dots,(n-1)/4\}\right\}.
  \end{align*}
  Our proof that the list for $n\equiv 1\pmod 4$ is complete follows closely the corresponding proof in~\cite{Gruettmueller2009}. Some small modifications are needed to take into account the possibility that $E_0\neq\emptyset$. The case $n=5$ can be treated by hand, and from now on we assume $n\geq 9$. In the following, we will formulate a sequence of claims providing more and more information about the structure of $H$, eventually allowing us to show that $H$ must be isomorphic to one of the graphs listed in the theorem. To keep the presentation of our main argument reasonably short, we postpone the proofs of the claims (some of which are a bit tedious)
  to \Cref{app:claims}. In view of \eqref{eq:final_inequality} we start from
  \begin{equation}\label{eq:n_1_base}
    \abs{E_0}+\alpha+\beta+\frac12\gamma\leq\frac{n}{2}.
  \end{equation}
  Let $M$ be the set of edges defined by
  \[M=\left\{xy\in E\,:\,N(x)\setminus\{y\}=N(y)\setminus\{x\}\right\}=\left\{xy\in E\,:\,t(xy)=\frac{d(x)+d(y)}{2}-1\right\}.\]
  \begin{restatable}{claim}{Mnonempty}\label{C:M_non-empty}
    If $M=\emptyset$ then $n=9$ and $H\cong H_9$.
  \end{restatable}
  From now on, we assume $M\neq\emptyset$. For $xy\in E$, set $V_1(xy)=N(x)\cap N(y)=\{z\,:\,xyz\in T\}$.
  \begin{restatable}{claim}{largedegreeweak}\label{C:large_degree_weak}
    For all $xy\in M$, $d(x)\geq\frac{n-1}{2}+\abs{E(V_1(xy))}$.
  \end{restatable}
  \begin{restatable}{claim}{largedegreestrong}\label{C:large_degree_strong}
    If there exists $xy\in M$ with $d(x)=\frac{n-1}{2}+\abs{E(V_1(xy))}$ then $H$ is isomorphic to $K^+_{(n+3)/2,(n-3)/2}$.
  \end{restatable}
  From now on we assume $d(x)\geq\frac{n+1}{2}+\abs{E(V_1(xy))}$ for every vertex $x$ incident with an edge in $M$.
  \begin{restatable}{claim}{vxyindep}\label{C:V1_xy_indep}
    $E(V_1(xy))=\emptyset$ for all $xy\in M$.
  \end{restatable}
  \begin{restatable}{claim}{inducedmatching}\label{C:induced_matching}
    $M$ is an induced matching with $\abs{M}\leq\frac{n-1}{4}$.
  \end{restatable}
  \begin{restatable}{claim}{largeM}\label{C:large_M}
    If $\abs{M}=\frac{n-1}{4}$ then $H$ is isomorphic to $K^+_{(n-1)/2,(n+1)/2}$ or
    $K^+_{(n+1)/2,(n-1)/2}$.
  \end{restatable}
  From now on we assume $\abs{M}\leq\frac{n-5}{4}$, and introduce the following notation:
  \begin{align*}
  V_1 &= \bigcup_{xy\in M}V_1(xy) & V_1' &= \bigcap_{xy\in M} V_1(xy) & V_2 &= V\setminus(V_1\cup V(M)).
  \end{align*}
  We also set $d_i(x)=\abs{\left\{y\in V_i\,:\,xy\in E\right\}}$ for $x\in V$ and $i\in\{1,2\}$.
  \begin{restatable}{claim}{largeintersection}\label{C:large_intersection}
    $\abs{V_1'}\geq\abs{V_1}-\frac{n-1}{4}$.
  \end{restatable}
  \begin{restatable}{claim}{voptions}\label{C:V1_options}
    $\abs{V_1}\in\left\{\frac{n-1}{2},\frac{n+1}{2}\right\}$.
  \end{restatable}
  \begin{restatable}{claim}{vindependent}\label{C:V1_independent}
    $V_1$ is an independent set in $H$, and $V_1=V_1'$.
  \end{restatable}
  \begin{restatable}{claim}{vsmall}\label{C:V1small}
    $\abs{V_1}=\frac{n-1}{2}$.
  \end{restatable}
  \begin{restatable}{claim}{lastclaim}\label{C:V2}
    $H[V_2]$ contains at most one isolated vertex and no isolated edge. In particular, $\abs{E(V_2)}\geq\frac23\left(\abs{V_2}-1\right)$.
  \end{restatable}
  Let $\delta=\abs{V_1}\abs{V_2}-\abs{E(V_1,V_2)}$ be the number of non-edges between $V_1$ and $V_2$.
  \begin{restatable}{claim}{vfewedges}\label{C:V2fewedges}
    $\delta\leq\frac{n-1}{2}-2$.
  \end{restatable}
  Finally, we can derive the required contradiction. By \Cref{C:V2fewedges}, there are $x,y\in V_1$ with $xz,yz\in E$ for all $z\in V_2$, hence $t(uv)\geq 2$ for all $uv\in E(V_2)$. Pick $x\in V_1$ with $d(x)=\frac{n+1}{2}$ and let $H'=H-x$ be the graph obtained from $H$ by deleting vertex $x$. With $E'$, $T'$ and $E'_0$ being the sets of edges, triangles, and edges not contained in a triangle in $H'$, we have
  \begin{multline*}
    \abs{E'}-\abs{T'}-\frac12\abs{E'_0}=\left(\abs{E}-\frac{n+1}{2}\right)-\left(\abs{T}-\abs{M}-\abs{E(V_2)}\right)-\frac12\abs{E_0}\\
    =\abs{E}-\abs{T}-\frac12\abs{E_0}-\frac{n+1}{2}+\abs{M}+\abs{E(V_2)}=\frac{n^2-2n-7}{8}+\abs{M}+\abs{E(V_2)}.
  \end{multline*}
  Using Claim~\ref{C:V2} and our assumption $\abs{M}\leq\frac{n-5}{4}$, we obtain
  \[\abs{M}+\abs{E(V_2)}\geq\abs{M}+\frac23\left(\frac{n+1}{2}-2\abs{M}-1\right)=\frac{n-1}{3}-\frac13\abs{M}\geq\frac{n-1}{3}-\frac{n-5}{12}=\frac{3n+1}{12}.\]
  But $\abs{M}+\abs{E(V_2)}$ is an integer, so we can round the right-hand side (taking into account that $n\equiv 1\pmod 4$): $\abs{M}+\abs{E(V_2)}\geq\frac{n+3}{4}$. Then
  \[\abs{E'}-\abs{T'}-\frac12\abs{E'_0}\geq\frac{n^2-2n-7}{8}+\frac{n+3}{4}=\frac{n^2-1}{8},\]
  and this implies $H'\cong K^+_{(n-1)/2,(n-1)/2}$, which is impossible because $H'$ contains a vertex of degree $\frac{n+1}{2}$ whose neighborhood is $V(M)\cup V_2$ while in $K^+_{(n-1)/2,(n-1)/2}$ the neighborhood of any vertex of degree $\frac{n+1}{2}$ induces a star $K_{1,(n-1)/2}$.
\end{proof}

%% file: 3-asymp.tex
\section{An upper bound for \texorpdfstring{$i(G_{k+1,k})$}{}}\label{sec:k_greater_than_2}
In this section we prove Theorem~\ref{thm:independent_dominating_asymp}. For $k=2$, the independent dominating sets in $G_{k+1,k}$ correspond to graphs with the property that every edge is contained in a triangle. We generalize this to $k$-uniform hypergraphs ($k$-graphs for short) in the following way.
\begin{definition}
    A $k$-graph is \emph{well-covered} if every edge is contained in a $(k+1)$-clique. In other words, for every edge $e=\{v_1,\dots,v_k\}$, there exists a vertex $v_{k+1}\in V(H)\setminus e$ such that $e_i=e\setminus\{v_i\}\cup\{v_{k+1}\}$ is an edge for every $i\in[k]$.
\end{definition}
As in the case $k=2$, there is a correspondence between well-covered $k$-graphs $H$ and independent dominating sets $D$ in $G_{k+1,k}$. The dominating set $D(H)$ has the $(k+1)$-cliques of $H$ as the vertices on level $k+1$, and the $k$-sets which are not edges of $H$ as the vertices on level $k$. Conversely, given an independent dominating set $D$ in $G_{k+1,k}$ we get a well-covered $k$-graph $H(D)$ by taking the $k$-sets which are not in $D$ as the edges. For a $k$-graph $H$, let $e(H)$ and $c(H)$ be the numbers of edges and of $(k+1)$-cliques in $H$, respectively. The above observation implies that finding $i(G_{k+1,k})$ is equivalent to maximizing $e(H)-c(H)$ over all well-covered $k$-graphs $H$. More precisely, $i(G_{k+1,k})=\binom{n}{k}-e(H)+c(H)$ for an optimal $H$.

Our approach is to generalize the graphs $K^+_{s,n-s}$ (for even $s$) from Theorem~\ref{thm:extremal_constructions}. One way of looking at this construction is as follows: In order to construct a graph which has many edges and few triangles subject to the condition that every edge is contained in a triangle, we partition the vertex set into two parts $A$ and $B$, take all edges with exactly one vertex in $B$ (and one in $A$), and add a set $M$ of edges in $A$ with the property that every $x\in A$ is contained in exactly one edge in $M$ (that is, a perfect matching on $A$). Translating this directly to $k$-graphs, we would like to take all $k$-sets with exactly one vertex in $B$ (and $k-1$ vertices in $A$), and then add a collection $M$ of $k$-subsets of $A$ such that every $(k-1)$-subset of $A$ is contained in exactly one member of $M$. The complete $k$-graphs on $k+1$ vertices are then precisely the sets $X\cup\{b\}$ with $X\in M$ and $b\in B$.    
\begin{example}\label{ex:first_step}
    Let $n=11$ and $k=3$. For $A=\{1,2,\dots,7\}$ and $B=\{8,9,10,11\}$, we can take $M$ to be a Steiner triple system on $A$ to obtain a 3-graph with $\abs{M}+\binom{\abs{A}}{2}\abs{B}=91$ edges and $\abs{M}\abs{B}=28$ complete $3$-graphs on $4$ vertices. The corresponding dominating set in $G_{4,3}$ has size $102$ ($28$ 4-sets and $\binom{11}{3}-91=74$ 3-sets).  
\end{example}
As the design condition on $M$ (every $(k-1)$-set is covered exactly once) is too much to hope for we modify the requirement as follows: We take $M$ to be a collection of $k$-subsets of $A$ such that every $(k-1)$-subset of $A$ at most once. As a consequence, not all $k$-sets with exactly one vertex in $B$ can be edges, but only those of the form $X\cup\{b\}$ with $b\in B$ and $X$ a $(k-1)$-subset of $A$ that is covered by a member of $M$. In the following lemma we determine a lower bound for the value $e(H)-c(H)$ that can be obtained using this construction. 
\begin{lemma}\label{lem:base_constuction}
    Fix $\alpha$ with $0<\alpha<1$, and set $A=\left\{1,2,\dots,\lfloor(1-\alpha)n\rfloor\right\}$, $B=[n]\setminus A$. There exists a well-covered $k$-graph $H$ such that
    \begin{enumerate}[(i)]
        \item $e(H)-c(H)\geq\left[(k-1)\alpha(1-\alpha)^{k-1}+o(1)\right]\binom{n}{k}$, and
        \item $\abs{X\cap B}\leq 1$ for every $X\in E(H)$.
    \end{enumerate}
\end{lemma}
\begin{proof}
    By \cite[Theorem 1]{Graham1980}, there exists $S\subseteq\binom{A}{k}$ with $\abs{X\cap Y}\leq k-2$ for all distinct $X,Y\in S$ and $\abs{S}\geq\frac{1}{\abs{A}}\binom{\abs{A}}{k}$. We define the edge set of $H$ by $E(H)=S\cup\{X\cup\{b\}\,:\,X\in\Delta S,\,b\in B\}$. Then $H$ is a $k$-graph satisfying (ii), and we can also verify (i):
    \begin{multline*}
    e(H)-c(H)=\abs{S}+k\abs{S}\abs{B}-\abs{S}\abs{B}=(k-1)\abs{S}\abs{B}+o(n^k)\\
    \geq\left[\frac{(k-1)\alpha}{1-\alpha}+o(1)\right]\binom{(1-\alpha)n}{k}=\left[(k-1)\alpha(1-\alpha)^{k-1}+o(1)\right]\binom{n}{k}.\qedhere
    \end{multline*}
\end{proof}
In contrast to the case $k=2$, for $k\geq 3$ the construction can still be improved. Note that so far the set $B$ is very thin: It contains at most 1 vertex from every edge. As a consequence, we can add another $k$-graph $H'$ on the vertex set $B$, and the graph $H''$ obtained by taking the union of the edge sets of $H$ and $H'$ satisfies $e(H'')-c(H'')=e(H)-c(H)+e(H')-c(H')$. A natural idea is to use the same construction for $H'$ as for $H$, and then this can be iterated.
\begin{example}\label{ex:second_step}
    For $n=30$ and $k=3$, we start with $A=\{1,2,\dots,19\}$ and $B=\{20,21,\dots,30\}$. Taking $M$ as a Steiner triple system on $A$, we obtain a 3-graph $H$ with $\frac13\binom{19}{2}+\binom{19}{2}\times 11=1938$ edges and $\frac13\binom{19}{2}\times 11=627$ complete $3$-graphs on 4 vertices. In the next stage, we split $B$ into $A_1=\{20,21,\dots,26\}$ and $B_1=\{27,28,29,30\}$, and as in Example~\ref{ex:first_step}, we obtain a 3-graph on vertex set $B$ with 91 edges and 28 complete 3-graphs on 4 vertices. Putting both parts together, we have $2029$ edges and 655 complete 3-graphs on 4 vertices. The corresponding dominating set has size  $2686\approx 0.66\binom{30}{3}$ (655 4-sets and $\binom{30}{3}-2029=2031$ 3-sets).
\end{example}
Now we want to make the above idea more precise. We apply the construction from \Cref{lem:base_constuction} recursively. Start with a partition $[n]=A_0\cup A_1\cup\dots\cup A_r$ such that
\begin{enumerate}[(i)]
    \item $\abs{A_0}=\left\lfloor(1-\alpha)n\right\rfloor$,
    \item $\abs{A_i}=\left\lfloor(1-\alpha)(n-\abs{A_0\cup\dots\cup A_{i-1}})\right\rfloor$ for $i=1,\dots,r$, and
    \item $\left\lfloor(1-\alpha)\abs{A_r}\right\rfloor<k\leq\left\lfloor(1-\alpha)(n-\abs{A_0\cup\dots\cup A_{r-2}})\right\rfloor$.
\end{enumerate}
For $i\in\{0,1,\dots,r-1\}$, let $H_i$ be the $k$-graph from \Cref{lem:base_constuction} on the vertex set $A_i\cup A_{i+1}\cup\dots\cup A_r$ with $A=A_i$ and $B=A_{i+1}\cup\dots\cup A_r$, and let $H$ be the $k$-graph on the vertex set $[n]$ with edge set $E(H)=E(H_0)\cup\dots\cup E(H_{r-1})$. In the following lemma we provide the asymptotics for $e(H)-c(H)$.
\begin{lemma}\label{lem:asymptotics}
  $\displaystyle e(H)-c(H)\geq\left[\frac{(k-1)\alpha(1-\alpha)^{k-1}}{1-\alpha^k}+o(1)\right]\binom{n}{k}$.  
\end{lemma}
\begin{proof}
    Fix $\varepsilon>0$ and then fix $\varepsilon_1>0$ satisfying
    \[\varepsilon_1\max\left\{\frac{(k-1)\alpha(1-\alpha)^{k-1}}{1-\alpha^k},\,\frac{1}{1-\alpha^k},\,\varepsilon_1\right\}<\frac{\varepsilon}{3}.\]
    Choose $N$ sufficiently large such that
    \[1+\alpha^k+\alpha^{2k}+\dots+\alpha^{NK}\geq\frac{1}{1-\alpha^k}-\varepsilon_1.\]
    By construction, $\abs{A_i\cup\dots\cup A_r}\geq\alpha^in$, and by \Cref{lem:base_constuction},
    \[e(H_i)-c(H_i)\geq\left((k-1)\alpha(1-\alpha)^{k-1}+o(1)\right)\binom{\alpha^in}{k}=\left((k-1)\alpha(1-\alpha)^{k-1}+o(1)\right)\alpha^{ik}\binom{n}{k}.\]
    Hence, if $n$ is sufficiently large, then for all $i\in\{0,1,\dots, N\}$,
    \[e(H_i)-c(H_i)\geq\left[(k-1)\alpha(1-\alpha)^{k-1}-\varepsilon_1\right]\alpha^{ik}\binom{n}{k}.\]
    Taking the sum over $i$, we obtain
    \begin{multline*}
      e(H)-c(H)\geq\sum_{i=0}^N\left(e(H_i)-c(H_i)\right)\geq\sum_{i=0}^N\left[(k-1)\alpha(1-\alpha)^{k-1}-\varepsilon_1\right]\alpha^{ik}\binom{n}{k}  \\
      \geq\left[(k-1)\alpha(1-\alpha)^{k-1}-\varepsilon_1\right]\left[\frac{1}{1-\alpha^k}-\varepsilon_1\right]\binom{n}{k} \geq\left[\frac{(k-1)\alpha(1-\alpha)^{k-1}}{1-\alpha^k}-\varepsilon\right]\binom{n}{k},
    \end{multline*}
    and this concludes the proof.
\end{proof}
Maximizing the right hand side in \Cref{lem:asymptotics} , we choose $\alpha$ to be the root of the polynomial $(k-1)x^k-kx+1$ in the interval $[0,1/2]$, and substituting $\frac{k\alpha-1}{k-1}$ for $\alpha^k$ yields
\[e(H)-c(H)=\left[\frac{(k-1)^2\alpha(1-\alpha)^{k-2}}{k}+o(1)\right]\binom{n}{k},\]
and this concludes the proof of \Cref{thm:independent_dominating_asymp}.

%% file: 4-open_problems.tex
\section{Conclusion and open problems}\label{sec:open}
In this paper, we have investigated the independent domination number of $G_{k+1,k}$. For $k=2$, we proved the exact value and a complete characterization of the smallest independent dominating sets, and for $k\geq 3$ we improved the asymptotic upper bound from~\cite{Gerbner2012}. It would be nice to close the gap between the upper and lower bound, so we restate the following problem from~\cite{Gerbner2012} in our notation.
\begin{problem}
    Try to close the gap between the lower bound in~\eqref{eq:gerbner_bound} and the upper bound in Theorem \ref{thm:independent_dominating_asymp}. In particular, find a lower bound that does not depend on hypergraph Tur\'an densities.
\end{problem}
Another natural direction for further investigations is to consider $G_{l,k}$ where $l$ and $k$ are not consecutive. It is not hard to see that the lower bound in~\eqref{eq:gerbner_bound} actually applies to $\gamma(G_{k+1,k})$ instead of $i(G_{k+1,k})$ and generalizes as follows:
\[\gamma(G_{l,k})\geq\left(1-\frac{\binom{l}{k}-2}{\binom{l}{k}-1}t_{l,k}-o(1)\right)\binom{n}{k},\]
with
\[t_{l,k}=\lim_{n\to\infty}\frac{\operatorname{ex}\left(n,K_{l}^{(k)}\right)}{\binom{n}{k}}.\]
One problem with this bound is that the exact values of $t_{l,k}$ are not known for $k\geq 3$. So it makes sense to focus on the case $k=2$ with $t_{l,2}=\frac{l-2}{l-1}$. In this case it has been proved that the above lower bound for $\gamma(G_{l,2})$ is tight~\cite{Balogh2021}:
\[\gamma(G_{l,2})=\left(1-\frac{\binom{l}{2}-2}{\binom{l}{2}-1}\frac{l-2}{l-1}+o(1)\right)\binom{n}{2}=\left(\frac{l+3}{(l-1)(l+1)}+o(1)\right)\binom{n}{2}.\]
For $l=3$ the smallest dominating set can be taken to be independent, but for $l\geq 4$ it looks plausible that $i(G_{l,2})>\gamma(G_{l,2})$. 
For instance, for $n=9$ it can be checked that $i(G_{4,2})=17$, while $\gamma(G_{4,2})=15$. Corresponding constructions are illustrated in Figure~\ref{fig:4-2}.
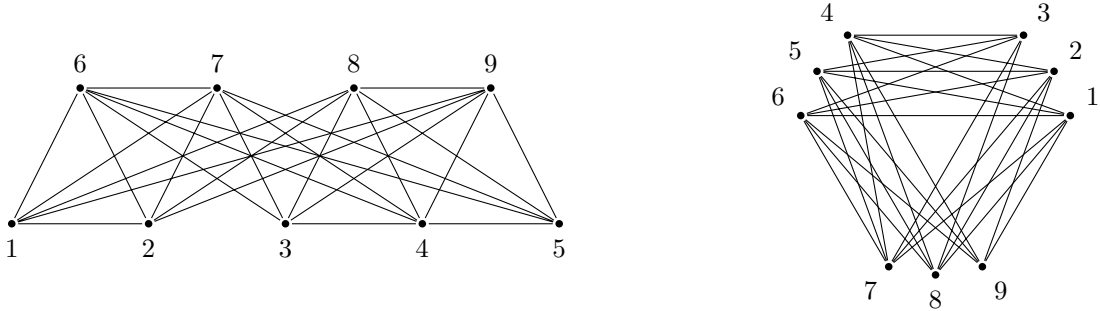
\begin{figure}[htb]
    \begin{minipage}{.49\textwidth}
    \centering
        \begin{tikzpicture}[scale=.9]
            \node[label={270:1},circle,fill=black,outer sep=1pt,inner sep=1pt] (v1) at (0,0) {};
            \node[label={270:2},circle,fill=black,outer sep=1pt,inner sep=1pt] (v2) at (2,0) {};
            \node[label={270:3},circle,fill=black,outer sep=1pt,inner sep=1pt] (v3) at (4,0) {};
            \node[label={270:4},circle,fill=black,outer sep=1pt,inner sep=1pt] (v4) at (6,0) {};
            \node[label={270:5},circle,fill=black,outer sep=1pt,inner sep=1pt] (v5) at (8,0) {};
            \node[label={90:6},circle,fill=black,outer sep=1pt,inner sep=1pt] (v6) at (1,2) {};
            \node[label={90:7},circle,fill=black,outer sep=1pt,inner sep=1pt] (v7) at (3,2) {};
            \node[label={90:8},circle,fill=black,outer sep=1pt,inner sep=1pt] (v8) at (5,2) {};
            \node[label={90:9},circle,fill=black,outer sep=1pt,inner sep=1pt] (v9) at (7,2) {};
            \draw (v1) -- (v2);
            \draw (v6) -- (v7);
            \draw (v8) -- (v9) -- (v5);
            \draw (v3) -- (v4) -- (v5);
            \draw (v1) -- (v6) -- (v2) -- (v7) -- (v3) -- (v6) -- (v4) -- (v7) -- (v1) -- (v8) -- (v2) -- (v9) -- (v1);
            \draw (v7) -- (v5) -- (v6);
            \draw (v5) -- (v8) -- (v3) -- (v9) -- (v4) -- (v8);            
        \end{tikzpicture}
    \end{minipage}\hfill
    \begin{minipage}{.49\textwidth}
    \centering
      \begin{tikzpicture}[scale=.9]
          \node[label={10:1},circle,fill=black,outer sep=1pt,inner sep=1pt] (v1) at (10:2) {};
          \node[label={30:2},circle,fill=black,outer sep=1pt,inner sep=1pt] (v2) at (30:2) {};
          \node[label={50:3},circle,fill=black,outer sep=1pt,inner sep=1pt] (v3) at (50:2) {};
          \node[label={130:4},circle,fill=black,outer sep=1pt,inner sep=1pt] (v4) at (130:2) {};
          \node[label={150:5},circle,fill=black,outer sep=1pt,inner sep=1pt] (v5) at (150:2) {};
          \node[label={170:6},circle,fill=black,outer sep=1pt,inner sep=1pt] (v6) at (170:2) {};
          \node[label={250:7},circle,fill=black,outer sep=1pt,inner sep=1pt] (v7) at (250:2) {};
          \node[label={270:8},circle,fill=black,outer sep=1pt,inner sep=1pt] (v8) at (270:2) {};
          \node[label={290:9},circle,fill=black,outer sep=1pt,inner sep=1pt] (v9) at (290:2) {};
          \draw (v1) -- (v4) -- (v7) -- (v1) -- (v5) -- (v7) -- (v2) -- (v4) -- (v8) -- (v1) -- (v6) -- (v7) -- (v3) -- (v4) -- (v9) -- (v1);
          \draw (v2) -- (v5) -- (v8) -- (v2) -- (v6) -- (v8) --  (v3) -- (v5) -- (v9) -- (v2);
          \draw (v3) -- (v6) -- (v9);
      \end{tikzpicture}  
    \end{minipage}
    \caption{Illustration for an optimal independent dominating set (left) and an optimal dominating set (right) in $G_{4,3}$ for $n=9$.}\label{fig:4-2}
\end{figure}
The independent dominating set is represented by the graph on the left: We take the vertex sets of the 6 copies of $K_4$ as 4-sets together with the 11 non-edges as 2-sets. For the general dominating set, we look at the complete tripartite graph $K_{3,3,3}$ shown on the right: The 4-sets $\{1,2,4,7\}$, $\{1,2,5,8\}$, $\{1,2,6,9\}$, $\{3,4,5,9\}$, $\{3,6,7,8\}$ and $\{4,5,7,8\}$ cover all the edges, and together with the 9 non-edges we obtain a dominating set of size 15. To verify optimality we solved corresponding binary linear programs using Gurobi \cite{gurobi}.
It would be interesting to establish an asymptotic difference between $i(G_{l,2})$ and $\gamma(G_{l,2})$ as follows.
\begin{problem}
Prove that for every $l\geq 4$ there exists $\varepsilon>0$ with 
\[i(G_{l,2})\geq\left(\frac{l+3}{(l-1)(l+1)}+\varepsilon-o(1)\right)\binom{n}{2}.\]
\end{problem}
For instance, for $l=4$, the best construction we are aware of is from~\cite{Kalinowski2013} and gives $i(G_{4,2})\leq\left(\frac58+o(1)\right)\binom{n}{2}$, while $\gamma(G_{4,2})=\left(\frac{7}{15}+o(1)\right)\binom{n}{2}$.

%% file: 5-appendix.tex
\begin{appendix}
  \section{Proofs for the claims in the proof of Theorem~\ref{thm:extremal_constructions}}\label{app:claims}
  Recall that we assume $n\equiv 1\pmod 4$ and $n>9$, and that for $n\not\equiv 1\pmod 4$, we have
  already established that the graphs listed in \Cref{thm:extremal_constructions} are the only
  graphs with $\abs{E}-\abs{T}-\frac12\abs{E_0}=\left\lfloor\frac{(n-1)^2}{8}\right\rfloor$.
  \Mnonempty*
  \begin{proof} Assume $M=\emptyset$. From
    \[\beta\geq\frac12\sum_{xy\in E_1}\left(t(xy)-1\right)=\frac32\abs{T}-\frac12\abs{E_1},\]
    \[\gamma=\sum_{x\in V}\left(\frac{n+1}{2}-d(x)\right)^2\geq\sum_{x\in
        V}\left(\frac{n+1}{2}-d(x)\right)=\frac{n(n+1)}{2}-2\abs{E}\]
    and~(\ref{eq:n_1_base}), we obtain
    \[n\geq 2\abs{E_0}+2\beta+\gamma\geq \frac{n(n+1)}{2}+3\left(\abs{T}-\abs{E_1}\right).\]
    Using the optimality of $H$ and \Cref{lem:main_inequality},
    \[\frac{(n+3)(n-1)}{8}=\left\lfloor\frac{(n+1)^2}{8}\right\rfloor=\abs{E}-\abs{T}-\frac12\abs{E_0}=\abs{E_1}-\abs{T}+\frac12\abs{E_0}\\
        \geq\frac{n(n-1)}{6}+\frac12\abs{E_0},\]
        which simplifies to $0\geq n^2-10n+9+12\abs{E_0}=(n-1)(n-9)+12\abs{E_0}$. Together with $n\geq 9$, this implies $n=9$ and $E_0=\emptyset$, and then $H\cong H_9$ by~\cite{Gruettmueller2009}.
  \end{proof}
  From now on, we assume $M\neq\emptyset$. For $xy\in E$, set $V_1(xy)=N(x)\cap N(y)=\{z\,:\,xyz\in T\}$.
  \largedegreeweak*
  \begin{proof}
    Fix $xy\in M$ and set $H'=H-\{x,y\}$. In $H'$, let $E'$, $T'$ and $E_0'$ be the set of edges, the set of
    triangles, and the set of edges not in a triangle, respectively. Then
    \begin{align*}
      \abs{E'} &= \abs{E}-2d(x)+1,\\
      \abs{T'} &= \abs{T}-(d(x)-1)-2\abs{E(V_1(xy))},\\
      \abs{E'_0} &= \abs{E_0} +\abs{\{uv\in E(V_1(xy))\,:\,t(uv)=2\}}\leq\abs{E_0}+\abs{E(V_1(xy))}.
    \end{align*}
    By \Cref{thm:23} applied to $H'$,
    \[\frac{(n-1)^2}{8}\geq\abs{E'}-\abs{T'}-\frac12\abs{E'_0}\geq\abs{E}-\abs{T}-\frac12\abs{E_0}-\left(d(x)-\abs{E(V_1(xy))}\right),\]
    and with $\abs{E}-\abs{T}-\frac12\abs{E_0}=\frac{(n-1)(n+3)}{8}$, the claim follows.
  \end{proof}
  \largedegreestrong*
  \begin{proof}
    By \cite[Theorem 3]{Gruettmueller2009}, equality in \Cref{C:large_degree_weak} implies that $H'$ is isomorphic to
    $K^+_{(n-1)/2,(n-3)/2}$, and it follows that $V_1(xy)$ is the independent set of size $(n-3)/2$,
    hence the claim.
  \end{proof}
  From now on we assume $d(x)\geq\frac{n+1}{2}+\abs{E(V_1(xy))}$
  \vxyindep*
  \begin{proof}
    Suppose $uv\in E(V_1(xy))$. Then $t(xu)-1\geq 1$, and with $N(x)\cap N(u)\subseteq V_1(xy)\cup\{y\}$,
    \[t(xu)\leq\abs{E(V_1(xy))}+1\leq d(x)-\frac{n+1}{2}+1=d(x)-\frac{n-1}{2},\]
    and then
    \[\frac{d(x)+d(u)}{2}-t(xu)-1\geq\frac{d(x)+t(xu)+1}{2}-t(xu)-1=\frac{d(x)-t(xu)-1}{2}\geq\frac{n-3}{4}.\]  
    Similarly, for the edges $xv$, $yu$ and $yv$, and therefore $\beta\geq n-3>\frac{n}{2}$, which contradicts~(\ref{eq:n_1_base}).
  \end{proof}
  \inducedmatching*
  \begin{proof}
    If $xy,yz\in M$ then $wz\in E(V_1(xy))$ for every $w\in V_1(xy)\setminus\{z\}$, which contradicts \Cref{C:V1_xy_indep}. This shows that $M$ is a matching. Now suppose $xy,zw\in M$
    and $xz\in E$. By definition of $M$, $yz,xw,yw\in E$, and then $zw\in E(V_1(xy))$, again contradicting \Cref{C:V1_xy_indep}. Therefore, $M$ is an induced matching. In particular, for
    $xy\in M$, $V_1(xy)\subseteq V\setminus V(M)$, and with $\abs{V_1(xy)}=d(x)-1\geq\frac{n-1}{2}$, it follows that $\abs{M}\leq\frac{n-1}{4}$.
  \end{proof}
  \largeM*
  \begin{proof}
    Recall the notation
    \begin{align*}
      V_1 &= \bigcup_{xy\in M}V_1(xy) & V_1' &= \bigcap_{xy\in M} V_1(xy) & V_2 &= V\setminus(V_1\cup V(M)).
    \end{align*}
    For $xy\in M$ it follows from $d(x)\geq\frac{n+1}{2}$, $V_1(xy)\subseteq V\setminus V(M)$ and $\abs{V(M)}=\frac{n-1}{2}$ that $\abs{V_1(xy)}\in\{\frac{n-1}{2},\frac{n+1}{2}\}$.
    \begin{description}
    \item[Case 1] $V\setminus V(M)$ is an independent set. The maximality of
      $\abs{E}-\abs{T}-\frac12\abs{E_0}$ implies that $V'_1=V\setminus V(M)$, hence $H$ is isomorphic to $K^+_{(n-1)/2,(n+1)/2}$.
    \item[Case 2] $\abs{V_1}=\frac{n-1}{2}$. Then $V_1'=V_1$, $\abs{V_2}=1$ (say $V_2=\{y\}$), and the maximality of $\abs{E}-\abs{T}-\frac12\abs{E_0}$ implies that $xy\in E$ for all $x\in V'_1$. Hence $H\cong K^+_{(n+1)/2,(n-1)/2}$.
    \item[Case 3] $\abs{V_1}=\frac{n+1}{2}$ but $\abs{V_1(xy)}=\frac{n-1}{2}$ for all $xy\in M$. There must be an edge in $V\setminus V(M)$ (otherwise we are in Case 1), but there can't
      be more than one edge because the sets $V_1(xy)$ are independent sets. So let's say $E(V\setminus V(M))=\{uv\}$. For every $xy\in M$, $V_1(xy)=V_1\setminus\{u\}$ or
      $V_1(xy)=V_1\setminus\{v\}$, hence
      \[\abs{E}-\abs{T}-\frac12\abs{E_0}=\left(\left(\frac{n-1}{2}\right)^2+\frac{n-1}{4}+1\right)-\frac{(n-1)^2}{8}-\frac12=\frac{n^2+3}{8}<\frac{n^2+2n-3}{8}.\qedhere\]
    \end{description}

  \end{proof}
  From now on we assume $\abs{M}\leq\frac{n-5}{4}$.
  \largeintersection*
  \begin{proof}
    For all $z\in V_1\setminus V'_1$, say $z\in V_1(xy)\setminus V_1(x'y')$, $N(z)\subseteq V\setminus (V_1(xy)\cup\{x',y'\})$, hence $d(z)\leq\frac{n-3}{2}$. Then it follows
    from~(\ref{eq:n_1_base}) that
    \[n\geq\gamma\geq\sum_{z\in V_1\setminus V'_1}\left(\frac{n+1}{2}-d(z)\right)^2\geq 4\abs{V_1\setminus V'_1},\]
    hence $\abs{V_1\setminus V'_1}\leq\lfloor n/4\rfloor=\frac{n-1}{4}$.
  \end{proof}
  \voptions*
  \begin{proof}
    $\abs{V_1}\geq\frac{n-1}{2}$ follows from $d(x)\geq\frac{n+1}{2}$ for all $xy\in M$. Suppose $\abs{V_1}\geq\frac{n+3}{2}$. Then \Cref{C:large_intersection} implies
    $\abs{V'_1}\geq\frac{n+7}{4}$, and by \Cref{C:V1_xy_indep}, $d(z)\leq n-\abs{V_1}\leq\frac{n-3}{2}$ for every $z\in V'_1$. Hence $\gamma\geq 4\abs{V'_1}\geq n+7$, which contradicts~(\ref{eq:n_1_base}).
  \end{proof}
  \vindependent*
  \begin{proof}
    Suppose $uv\in E(V_1)$. Then $\abs{V_1}=\frac{n+1}{2}$ and for every $xy\in M$, $V_1(xy)=V_1\setminus\{u\}$ or $V_1(xy)=V_1\setminus\{v\}$, and by \Cref{C:V1_xy_indep},
    $E(V_1)=\{uv\}$. For every $xy\in M$, $z\in V'_1= V_1\setminus\{u,v\}$, either $E\cap\{xv,yv,zv\}=\emptyset$ or $E\cap\{xu,yu,zu\}=\emptyset$, hence
    \[\alpha\geq\abs{M}\abs{V'_1}=\abs{M}\frac{n-3}{2}\geq\frac{n-3}{2}.\]
    Moreover, for every $z\in V'_1$, $d(z)\leq\frac{n-1}{2}$ and also $\min\{d(u),d(v)\}\leq\frac{n-3}{2}$, hence
    $\gamma\geq\abs{V'_1}+4=\frac{n+5}{2}$. With~(\ref{eq:n_1_base}) this implies
    $\frac{n}{2}\geq\frac{n-3}{2}+\frac{n+5}{4}$, hence $n\leq 1$. Finally, $V_1(xy)=V_1$ for all $xy\in M$ follows from the maximality of $\abs{E}-\abs{T}-\frac12\abs{E_0}$.
  \end{proof}
  Let $\delta=\abs{V_1}\abs{V_2}-\abs{E(V_1,V_2)}$ be the number of non-edges between $V_1$ and $V_2$.
  \vsmall*
  \begin{proof}
      Suppose $\abs{V_1}=\frac{n+1}{2}$. From $d(x)=\frac{n+3}{2}$ for all $x\in V(M)$, it follows that
    \begin{equation}\label{eq:gamma_bound_1}
      \sum_{x\in V(M)}\left(\frac{n+1}{2}-d(x)\right)^2=2\abs{M}.
    \end{equation}
    By~\Cref{C:V1_independent}, $d(x)\leq\frac{n-1}{2}$ for every $x\in V_1$. Removing an edge $xy$ with $x\in V_1$ and $y\in V_2$ increases $(\frac{n+1}{2}-d(x))^2$ by at least $3$, hence
    \begin{equation}\label{eq:gamma_bound_2}
      \sum_{x\in V_1}\left(\frac{n+1}{2}-d(x)\right)^2\geq\frac{n+1}{2}+3\delta.
    \end{equation}
    Recall that for $i\in\{1,2\}$, $d_i(x)$ is the number of neighbors of $x$ in $V_i$.
    If $\delta=0$ then (using that $d_2(x)\geq 2$ for at least one $x\in V_2$ by \Cref{C:V2})
    \[\sum_{x\in V_2}\left(\frac{n+1}{2}-d(x)\right)^2=\sum_{x\in V_2}d_2(x)^2\geq 2\abs{E(V_2)}+3\]
    and together with~(\ref{eq:gamma_bound_1}) and (\ref{eq:gamma_bound_2}),
    \[\gamma\geq\frac{n+1}{2}+2\abs{M}+2\abs{E(V_2)}+3\geq n+3.\]
    This contradicts~(\ref{eq:n_1_base}), and we can assume $\delta\geq 1$. For $x\in V_2$, let $\delta_x=\frac{n+1}{2}-d_1(x)$ be the number of $y\in V_1$ with $xy\not\in E$. Then $\left(\frac{n+1}{2}-d(x)\right)^2\geq d_2(x)-\delta_x$, hence
    \[\sum_{x\in V_2}\left(\frac{n+1}{2}-d(x)\right)^2\geq 2\abs{E(V_2)}-\delta,\]
    and together with~(\ref{eq:gamma_bound_1}) and (\ref{eq:gamma_bound_2}),
    \[\gamma\geq\frac{n+1}{2}+3\delta+2\abs{M}+2\abs{E(V_2)}-\delta\geq n+2\delta>n,\]
    which is the required contradiction.
  \end{proof}
  \lastclaim*
  \begin{proof}
    Suppose $x$ and $y$ are isolated vertices in $H[V_2]$. Then $\abs{E}-\abs{T}-\frac12\abs{E_0}$ can be increased by adding the edge $xy$ (and all the edges between $\{x,y\}$ and $V_1$ that
    might be missing in $H$). If $xy$ is an isolated edge in $H[V_2]$ then $V_1(xy)=V_1$ (by optimality of $H$), hence $xy\in M$.
  \end{proof}  
  \vfewedges*
  \begin{proof}
    For $xy\in M$ and $z\in V_1$, $\alpha_{xyz}\geq\abs{\{w\in V_2\,:\,zw\not\in E\}}$, hence $\alpha\geq\abs{M}\delta\geq\delta$. Moreover,
    \[\gamma\geq\sum_{x\in V_1}\left(\frac{n+1}{2}-d(x)\right)^2\geq\sum_{x\in V_1}\abs{\{y\in V_2\,:\,xy\not\in E\}}^2\geq\delta.\] 
    Now~(\ref{eq:n_1_base}) implies $\delta+\frac12\delta\leq\frac{n}{2}$. For $n\geq 13$, this already implies the claim: $\delta\leq\lfloor\frac{n}{3}\rfloor\leq\frac{n-5}{2}$. For $n=9$, we still have to rule out that $\delta=3$. It follows from Claims~\ref{C:V1small} and~\ref{C:V2}, together with $1\leq\abs{M}\leq\frac{n-5}{4}$, that $\abs{M}=1$, $\abs{V_2}=3$, and $V_2$ induces either a triangle or a path with two edges. In the first case ($V_2$ induces a triangle), there is no edge between $V(M)$ and this triangle, hence $\alpha\geq\delta+2$, and consequently $\delta\leq\lfloor\frac{9}{2}-2\rfloor=2$. In the second case ($V_2$ induces a path with two edges), we have $\sum_{x\in V_2}\left(\frac{n+1}{2}-d(x)\right)^2\geq \delta-1$, hence $\gamma\geq 2\delta-1$, and $\alpha+\frac12\gamma\geq 2\delta-\frac12$, which, together with~(\ref{eq:n_1_base}), implies $\delta\leq 2$.  
  \end{proof}
\end{appendix}